\documentclass[a4paper]{amsart}
\usepackage{amscd,amssymb}

\usepackage[all]{xy}

\renewcommand{\mod}{\operatorname{mod}\nolimits}

\newcommand{\add}{\operatorname{add}\nolimits}
\newcommand{\repdim}{\operatorname{repdim}\nolimits}

\newcommand{\Hom}{\operatorname{Hom}\nolimits}
\newcommand{\End}{\operatorname{End}\nolimits}

\newcommand{\Ker}{\operatorname{Ker}\nolimits}

\newcommand{\rrad}{\mathfrak{r}}
\newcommand{\rad}{\operatorname{rad}\nolimits}

\newcommand{\Soc}{\operatorname{Soc}\nolimits}

\newcommand{\gldim}{\operatorname{gldim}\nolimits}

\newcommand{\TrD}{\operatorname{TrD}\nolimits}

\newcommand{\op}{{\operatorname{op}\nolimits}}

\newcommand{\pd}{{\operatorname{pd}\nolimits}}

\newcommand{\comp}{\operatorname{\scriptstyle\circ}}

\renewcommand{\L}{\Lambda}
\newcommand{\Z}{{\mathbb Z}}

\renewcommand{\S}{{\mathcal S}}

\newcommand{\extto}{\xrightarrow}

\newtheorem{lem}{Lemma}[section]
\newtheorem{prop}[lem]{Proposition}
\newtheorem{cor}[lem]{Corollary}
\newtheorem{thm}[lem]{Theorem}
\theoremstyle{definition}

\newtheorem*{remark}{Remark}

\begin{document}

\title[Auslander generators of the exterior algebra]{The Auslander generators of the exterior algebra in two variables}
\author[Lada]{Magdalini Lada}
\thanks{The author was partially supported by NFR Storforsk grand No. 167130.}
\address{Magdalini Lada\\Institutt for matematiske fag\\
NTNU\\ N--7491 Trondheim\\ Norway} \email{magdalin@math.ntnu.no}

\date{\today}
\maketitle

\begin{abstract}We compute a complete set of non-isomorphic minimal Auslander generators for the exterior algebra in two variables.
\end{abstract}

\section*{Introduction}

In 1971, Auslander introduced a new homological dimension for artin algebras, called \emph{representation dimension}, which was meant to measure how far an algebra is from having finitely many isomorphism classes of finitely generated indecomposable modules ~\cite{A}. Quite a few years later, there were several publications that motivated the further investigation of the representation dimension. Among those we have ~\cite{Iy}, where Iyama proves that the representation dimension is always finite, ~\cite{IT}, where Igusa and Todorov prove that an artin algebra with representation dimension less than or equal to three satisfies the \emph{finitistic dimension conjecture} and ~\cite{R1}, where Rouquier uses the exterior algebras as examples of artin algebras with arbitrarily large representation dimension.

The last years, there has been an increasing interest in the topic and several researchers have worked on determining the representation dimension of certain classes of algebras. This is usually done by constructing a module, which is a \emph{generator-cogenerator} for the module category and is such that the global dimension of its endomorphism ring is the smallest among the global dimensions of the endomorhsim rings of all modules that are generators-cogenerators for the module category. Such a module is called an \emph{Auslander generator} and in general it is not easy to find. 

Not much is known about the class of Auslander generators of an artin algebra. In ~\cite{L2}, the author showed how we can construct, under certain conditions, a new Auslander generator by mutating a given one. In this paper, we compute all \emph{minimal} Auslander generators for the exterior algebra in two variables, up to isomorphism. This is, to our knowledge, the first non-trivial example where a complete set of minimal Auslander generators is computed.

\section{Background}
Let $\L$ be an artin algebra. We denote by $\mod \L$ the category of finitely generated left $\L$-modules, and by a $\L$-module we will always mean a module in $\mod\L$. For a $\L$-module $M$  we denote by $\add M$ the full subcategory of $\mod\L$, consisting of direct summands of copies of $M$. 

A $\L$-module $M$ is called a \emph{generator-cogenerator} if all the indecomposable projective and the indecomposable injective $\L$-modules are in $\add M$. The \emph{representation dimension} of $\L$, which we denote by $\repdim\L$, is defined as follows
\[\repdim\L=\inf \{\gldim\End_\L(M) \mid M  \ \rm{generator-cogenerator \  for}  \ \mod\L\}.\]
\sloppy A basic $\L$-module $M$ is called an \emph{Auslander generator} if $\gldim\End_\L(M)=\repdim\L$. An Auslander generator is called \emph{minimal} if for any direct summand $N$ of $M$, that does not contain any projective or injective $\L$-modules, we have $\gldim\End_\L(M/N)> \gldim\End_\L(M)$. Note that here, the factor module $M/N$ is the cokernel of the split monomorphism $N\hookrightarrow M$. 

The following result was proved in ~\cite{D}. Recall that, for an artin algebra $\L$, a \emph{node} is a non-projective, non-injective simple $\L$-module $S$, such that the middle term of the almost split sequence starting at $S$ is projective.
\begin{prop}\label{dugas}
Let $\L$ and $\L'$ be two artin algebras with no nodes, and $\alpha\colon \underline{\mod}\L\to\underline{\mod}\L'$ a stable  equivalence. If  the $\L$-module $\L\oplus N$ is an Auslander generator of $\L$, then the $\L'$-module $\L'\oplus \alpha N$ is an Auslander generator of $\L'$.
\end{prop}

As a straightforward consequence of  Proposition \ref{dugas}, we get the following corollary. Note that $\tau$ denotes the Auslander-Reiten translation and $\Omega$ denotes the syzygy functor. 
\begin{cor}\label{dugascor}
Let $\L$ be a selfinjective algebra with no nodes and $\L\oplus N$ an Auslander generator of $\L$. Then, the $\L$-modules $\L\oplus \tau N$ and $\L\oplus\Omega(N)$ are also Auslander generators.
\end{cor}
\begin{proof}
Since $\L$ is selfinjective, both of the functors $\tau$ and $\Omega$ induce a stable equivalence on $\mod\L$.
\end{proof}

\section{The exterior algebra}

In this section, $\L$ will denote the exterior algebra in two variables. We begin by describing $\L$ as the path algebra of a quiver modulo relations. Let $Q$ be the quiver 
\[\xymatrix@C=20pt{1\ar@(ul,dl)_\alpha\ar@(ur,dr)^\beta}\]
and let $kQ$ be the path algebra of $Q$ over some algebraically closed field $k$. Set $\L=kQ/I$, where $I$ is the ideal of $kQ$ generated by  $\{\alpha^2, \alpha\beta+\beta\alpha, \beta^2\}$.
The quotient $\L/\Soc \L$ is stably equivalent
to the Kronecker algebra (see \cite{ARS}), and using this we can describe the AR-quiver of
$\L$ as follows. For each $p$ in $\mathbb P^1(k)$ there is a tube of rank one

\[\xymatrix@C=0.6cm@R=0.2cm{&& \vdots& & \vdots &&\\
&R_p(5)\ar@{->>}[rd]& & R_p(5)\ar@{->>}[rd] &
& R_p(5)&\\
&& R_p(4)\ar@{->>}[rd]\ar@{^{(}->}[ru]& & R_p(4)\ar@{->>}[rd]\ar@{^{(}->}[ru]&&\\
\cdots & R_p(3)\ar@{->>}[rd]\ar@{^{(}->}[ru]& & R_p(3)\ar@{->>}[rd]\ar@{^{(}->}[ru] & &
R_p(3)&
\cdots \\
& & R_p(2)\ar@{->>}[rd]\ar@{^{(}->}[ru]& & R_p(2)\ar@{->>}[rd]\ar@{^{(}->}[ru]& & \\
&R_p(1)\ar@{^{(}->}[ru]&&R_p(1)\ar@{^{(}->}[ru]&&R_p(1)&\\}
\]

For $p=(1,\lambda)$ the indecomposable module $R_{(1,\lambda)}(n)$, which we will denote by $R_\lambda(n)$ corresponds to the
representation
\[\xymatrix@C=50pt{k^{2n}\ar@(ul,dl)_{f_\alpha^\lambda}\ar@(ur,dr)^{f_\beta^\lambda}}\]
where $f_\alpha^\lambda$ is given by the matrix 
$\left(\begin{smallmatrix}{\mathbf0}_n& {\mathbf0}_n \\ {\mathbf I}_n & {\mathbf0}_n\end{smallmatrix}\right)$ and $f_\beta^\lambda$ is given by the matrix $\left(\begin{smallmatrix}{\mathbf 0}_n& {\mathbf 0}_n \\ {\mathbf J}_n(\lambda) & {\mathbf 0}_n \end{smallmatrix}\right)$. Here ${\mathbf J}_n(\lambda)$ denotes the $n\times n$ Jordan block with eigenvalue $\lambda$. 

For $p=(0,1)$, the indecomposable module
$R_{(1,0)}(n)$, which we will denote by $R(n)$, corresponds to the representation
\[\xymatrix@C=50pt{k^{2n}\ar@(ul,dl)_{f_\alpha}\ar@(ur,dr)^{f_\beta}}\]
where $f_\alpha$ is given by the zero matrix $\mathbf 0_{2n}$ and $f_\beta$ is given by the matrix $\left(\begin{smallmatrix} {\mathbf 0}_n& {\mathbf0}_n \\ {\mathbf I}_n & {\mathbf 0}_n \end{smallmatrix}\right)$.


Besides the tubes, there is one more component that contains the projective-injective
 $\L$-module $\L$ and the simple $\L$-module $S$.

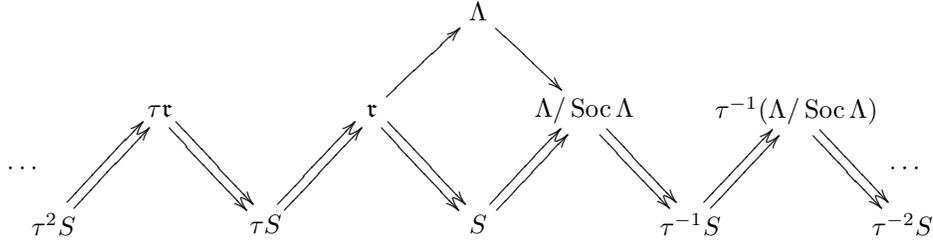
\[\xymatrix@C=-0.6cm@R=0.3cm{&&&&&\makebox[8mm]{$\L$}\ar[ddr]&&&&&&\\
&&&&&&&&&&&\\
&&\makebox[2.4cm]{$\tau\rrad$}\ar@<0.4ex>[ddr]\ar@<-0.4ex>[ddr]&&\makebox[2.4cm]{$\rrad$}\ar@<0.4ex>[ddr]\ar@<-0.4ex>[ddr]\ar[uur]&&\makebox[2.4cm]{$\L/\Soc \L$}\ar@<0.4ex>[ddr]\ar@<-0.4ex>[ddr] & & \makebox[2.4cm]{$\tau^{-1}( \L/\Soc\L)$} \ar@<0.4ex>[ddr]\ar@<-0.4ex>[ddr]&&&\\
\cdots&&&&&&&&&&&\cdots\\
&\makebox[1.2cm]{$\tau^2 S$}\ar@<0.4ex>[uur]\ar@<-0.4ex>[ruu]&&\makebox[1.2cm]{$\tau S$}\ar@<0.4ex>[uur]\ar@<-0.4ex>[ruu]&&\makebox[1.2cm]{$S$}\ar@<0.4ex>[uur]\ar@<-0.4ex>[ruu]&&\makebox[1.2cm]{$\tau^{-1} S$} \ar@<0.4ex>[uur]\ar@<-0.4ex>[ruu]&&\makebox[1.2cm]{$\tau^{-2} S$}&&}
\]
where by $\rrad$ we denote the radical of $\L$. For simplicity, set $S_{2m}=(\TrD)^m S$ and $S_{2m+1}=(\TrD)^m(\L/\Soc\L)$, $m\in \Z$.
For $n>0$, the $\L$-module $S_n$ corresponds to the representation
\[\xymatrix@C=50pt{\makebox[0.5cm]{ $k^{2n+1}$} \ar@(ul,dl)_-{g_\alpha}\ar@(ur,dr)^-{g_\beta}}\]
where $g_\alpha$ is given by the matrix $\left(\begin{smallmatrix} \mathbf 0_n& \mathbf 0_{n\times (n+1)} \\ \mathbf I_n & \mathbf 0_{n\times (n+1)}\\ {\mathbf 0}_{1\times n} & \mathbf 0_{1\times (n+1)}\end{smallmatrix}\right)$ and $g_{\beta}$ is given by the matrix $\left(\begin{smallmatrix} \mathbf 0_n& \mathbf 0_{n\times (n+1)} \\ \mathbf 0_{1\times n} & \mathbf 0_{1\times (n+1)}\\ \mathbf I_n & \mathbf 0_{n\times (n+1)} \end{smallmatrix}\right)$.
Dually for $n<0$, the $\L$-module $S_n$ corresponds to the representation
\[\xymatrix@C=50pt{\makebox[0.5cm]{ $k^{2n+1}$}  \ar@(ul,dl)_-{h_\alpha}\ar@(ur,dr)^-{h_\beta}}\]
where $h_\alpha = g_\alpha^T$ and $h_\beta = g_\beta ^T$. 

We already know by ~\cite{A} that $\repdim\L=3$ and that the $\L$-module $M_0=\L\oplus S_0\oplus \L/\Soc\L$ is an Auslander generator. In the next proposition we give an infinite set of nonisomorphic Auslander generators. Then we prove that this set forms a complete set of non-isomorphic minimal Auslander generators of $\L$. 
\begin{prop}
With the above notation, let $M_n=\L\oplus S_n\oplus S_{n+1}$. Then, $M_n$ is a minimal Auslander generator for all $n$ in $\Z$.
\end{prop}

\begin{proof}
For any integer $m$, we have by definition 
\begin{align}M_{2m} & = \L\oplus (\tau^{-1})^m S \oplus (\tau^{-1})^{m} (\L/\Soc \L)\notag\\ & = \L \oplus (\tau^{-1})^m(S\oplus \L/\Soc\L)\notag
\end{align}
and 
\begin{align}
M_{2m+1} & =\L\oplus (\tau^{-1})^m (\L/\Soc \L) \oplus (\tau^{-1})^{m+1} S\notag\\ 
& = \L\oplus (\tau^{-1})^m\Omega^{-1} S \oplus (\tau^{-1})^m\Omega(\L/\Soc\L)\notag \\ 
& = \L\oplus(\tau^{-1})^m\Omega^{-1}(S\oplus \L/\Soc\L) \notag .
\end{align}
Hence, by Corollary \ref{dugascor}, we only need to show that $M_0=\L\oplus S_0\oplus S_1=\L\oplus S \oplus \L/\Soc\L$ is a minimal Auslander generator. We already know, by ~\cite{A} that $M_0$ is an Auslander generator. As for the minimality, straighforward computations show that 
\begin{align}
\gldim\End_\L(M_0/S_0) & =\gldim\End_\L(M_0/S_1)\notag \\ &=\gldim\End_\L(M_0/(S_0\oplus S_1))\notag \\ & =\infty\notag.
\end{align} 
\end{proof}

\begin{remark}
We note that for $n\geq 0$ the Auslander generators $M_n$ can be obtained from $M_0$ by iterated mutation, as described in ~\cite[Section 4]{L2}. Dually for $n\leq 0$, the Auslander generators $M_n$ can be obtained from $M_{-1}$ by iterated mutation.
\end{remark}

The rest of the section is devoted to showing that the $\L$-modules $M_n$, for $n$ in $\Z$, are all the Auslander generators for $\L$, up to isomorphism. 

\smallskip

In the next proposition we compute the global dimension of the endomorphism ring of a generator of $\L$, whose nonprojective summands belong to some tube (not necessarily the same) of the AR-quiver of $\L$.  

We note here that the $\Hom$-spaces between any two indecomposable $\L$-modules can be described using the stable equivalence of $\L/\Soc \L$ and the Kronecker algebra. The $\Hom$-spaces between any two indecomposable modules of the Kronecker algebra have been analytically described in ~\cite[VIII.7]{ARS} and we refer the reader there for details. The only extra morphisms that exist between two indecomposable $\L$-modules of Loewy length 2, are the morphisms that factor through the simple $\L$-module. 

\begin{prop}\label{allreg}
Let $M$ be a generator for $\mod\L$ such that all its non-projective indecomposable summands are of the form $R_p(n)$ for some $p\in\mathbb P^1(k)$ and some $n\in \mathbb N$. Then $\gldim\End_\L(M)=\infty$.
\end{prop}

\begin{proof}
Let $p\in\mathbb P^1(k)$ such that there is an indecomposable direct summand of $M$
isomorphic to $R_p(n)$. Choose $n$ to be the largest natural number such that $R_p(n)$ is isomorphic to
a direct summand of $M$. We show by induction that on the quiver of $\End_\L(M)^{\op}$
there is a loop at the vertex corresponding to the simple module $\S_{R_p(n)}$, where by $\S_{R_p(n)}$ we denote the top of the indecomposable projective $\End_\L(M)^{\op}$-module $\Hom_\L(M, R_p(n))$. In
particular we show that there is a morphism $f_n\colon R_p(n)\to R_p(n)$, which is not
the identity, that does not factor through any indecomposable summand of $M$. Considering
the previous representation of the module $R_p(n)$, the morphism $f_n$ is given by the matrix 
\[A_{f_n}=\left(
\begin{array}{ccccc}
0&0&\cdots&0&0\\
\vdots&\vdots&&\vdots&\vdots\\
0&0&\cdots&0&0\\
1&0&\cdots&0&0
\end{array}\right)
\] 
If $M$ has an indecomposable direct summand $R_p(m)$, then by the choice of $n$, we have $m\leq n$ and it is not hard to see, looking at all possible morphsisms between $R_p(m)$ and $R_p(n)$, that in this case $f_n$ does not factor through $R_p(m)$. For $p$ different from $q$ and for any $n$ and $m$, the only nonzero morphsims from $R_p(n)$ to $R_q(m)$, and vice versa, factor through the simple $\L$-module $S$. Using this fact it is easy to see that $f_n$ does not factor through any other indecomposable summand of $M$ of
the form $R_q(m)$,  for any $m$ and $q\neq p$. We show by induction on $n$ that $f_n$ does not factor through $\L$ either.

Let $n=1$. It is not hard to see that the composition of any morphism  from $\Hom_\L(R_p(1),\L)$ with any morphism from $\Hom_\L(\L,R_p(1))$ is zero. It
follows that $f_1\colon R_p(1)\to R_p(1)$ does not factor through $\L$. Assume that
$f_k\colon R_p(k)\to R_p(k)$ does not factor through $\L$. We have the following
factorization of $f_k$
\[
\xymatrix{R_p(k)\ar[r]^{f_k}\ar@{^{(}->}[d]^{i}& R_p(k)\\
R_p(k+1)\ar[r]^{f_{k+1}}&R_p(k+1)\ar@{->>}[u]^{\pi}}\]
where $i$ and $\pi$ are the natural inclusion and natural projection respectively. Thus, we see that if $f_{k+1}$ factors through $\L$ then so does $f_k$. Hence $f_n$ does not factor through $\L$ for any $n$ in $\mathbb N$. Hence we have proven that the quiver of $\End_\L(M)$, has a loop at the vertex of the simple module corresponding to $R_p(n)$. This implies that $\gldim\End_\L(M)=\infty$ ~\cite{Ig}.
\end{proof}

The above proposition shows that the set of the indecomposable non-projective summands of an Auslander generator of $\L$ must contain at least one module isomorphic to $S_n$ for some $n$. Next, we consider the case where a generator of $\L$ has exactly one indecomposable direct summand isomorphic to $S_n$ for some integer $n$. 

\begin{prop}\label{oneprep}
Let $M$ be a generator for $\mod\L$ such that it has exactly one indecomposable direct summand isomorphic to $S_n$ for some integer $n$. Then $\gldim\End_\L(M)\geq 4$.
\end{prop}

\begin{proof}
Due to Corollary \ref{dugascor}, we can assume that $S_n=S_0=S$. Moreover, since $\gldim\End_\L(\L\oplus S)=\infty$, we can also assume that $M$ has at least one
more nonprojective indecomposable direct summand besides $S$. Set $\Gamma=\End_\L(M)^{\op}$. We will compute the projective dimension of the simple $\Gamma$-module $\S_S$, that corresponds to $S$. Let
$p_1,p_2,\ldots,p_m$ in $\mathbb P^1(k)$, be such that there is an indecomposable direct summand of $M$ isomorphic to $R_{p_i}(n_i)$. We choose $n_i$
to be the largest natural number such that $R_{p_i}(n_i)$ is isomorphic to a direct summand of $M$. Straightforward computations show that $\gldim \End_\L(\L\oplus S \oplus R_p(1))=\infty$, for any $p$ in $\mathbb P^1(k)$, so we can assume in addition that if $m=1$, then $n_1> 1$. Let $g_i\colon R_{p_i}(n_i)\to S$ be the morphism given by the matrix 
\[A_{g_i}=\left(\begin{array}{cccc}1&0&\cdots&0
\end{array}\right)\]
We show that we then have an exact
sequence
\[0\to S_{(\sum_{i=1}^{m}n_i)-1}\to R_{p_1}(n_1)\oplus\cdots
\oplus R_{p_m}(n_m)\extto{\underline{g}=(g_1,\ldots,g_m)}S\to 0\]
We use induction on $m$. Let $m=1$. We show  induction on $n_1$ that we have an exact sequence
\[0\to S_{n_1-1}\to R_{p_1}(n_1)\extto{g_1} S\to 0.\]
For $n_1=1$ it is obvious that the sequence 
\[0\to S\to R_{p_1}(1)\extto{g_1} S\to 0\]
is exact. Assume that for $n_1=k$ the sequence 
\[0\to S_{k-1}\to R_{p_1}(k)\extto{g_1} S\to 0\]
is exact and consider the following pushout diagram
\[\xymatrix@C=20pt{&0\ar[d]&0\ar[d]&&\\
0\ar[r]& S_{k-1}\ar[r]\ar[d]&R_{p_1}(k)\ar[r]\ar@{^{(}->}[d]^{i}& S\ar@{=}[d]\ar[r]& 0\\
0\ar[r]& \Ker g_1 \ar[d]\ar [r]& R_{p_1}(k+1)\ar[r]^-{g_1}\ar@{->>}[d]& S\ar[r]& 0\\
& R_{p_1}(1)\ar@{=}[r]\ar[d]& R_{p_1}(1)\ar[d]&&\\
&0&0&&}
\]
where the inclusion $i$ is the irreducible morphism from $R_{p_1}(k)$ to $R_{p_1}(k+1)$. Since the rightmost vertical short exact sequence of the diagram is non-split, the leftmost vertical short exact sequence is also non-split. We claim $\Ker g_1\simeq S_{k}$. First note that since all the morphisms involved in the above diagram are graded, with respect to the grading induced by the radical layers, we can view the above diagram over the Kronecker algebra. If $X$ is an indecomposable direct summand of $\Ker g_1$, then as we see from the leftmost vertical sequence of the diagram, there is a nonzero morphism from $S_{k-1}$, which is a preprojective module, to $X$, and a nonzero morphism from $X$ to $R_{p_1}(1)$, which is a regular module. This means that $X$ has to be either a preprojective module $S_t$ with $t\geq k-1$ or a regular module from the same tube as $R_{p_1}(1)$. Since the sequence $0\to S_{k-1}\to \Ker g_1\to R_{p_1}(1)\to 0$ does not split, it is easy to see, by counting the dimensions of the top and the socle sequence, that $X$ has to be isomorphic to $S_k$. Hence $\Ker g_1\simeq S_k$.

Next assume that for $m=k>1$ there is an exact sequence
\[0\to S_{(\sum_{i=1}^{k}n_i)-1}\to R_{p_1}(n_1)\oplus\cdots
\oplus R_{p_k}(n_k)\extto{\underline{g}=(g_1,\ldots,g_k)}S\to 0\]
Let $j\in\{1,\ldots,k,k+1\}$ and consider the following commutative exact  diagram: 
 
\[\xymatrix@C=25pt{&0\ar[d]&0\ar[d]&&\\
0\ar[r]& S_{(\sum_{\substack{i=1 \\ i\neq j}}^{k+1} n_i)-1}\ar[r]\ar[d]&\sum_{\substack{i=1\\ i\neq j}}^{k+1}R_{p_i}(n_i)\ar[r]\ar@{^{(}->}[d]& S\ar@{=}[d]\ar[r]& 0\\
0\ar[r]& \Ker \underline g \ar[d]\ar [r]&\sum_{i=1}^{k+1}R_{p_i}(n_i) \ar[r]^-{\underline g}\ar@{->>}[d]& S\ar[r]& 0\\
& R_{p_j}(n_j)\ar@{=}[r]\ar[d]& R_{p_j}(n_j)\ar[d]&&\\
&0&0&&}
\]

We view the above diagram over the Kronecker algebra $H$. Then, as we see from the diagram, there is a nonzero morphism form $\Ker \underline g$ to $R_{p_j}(n_j)$, for all $j$ in $\{1,\ldots,k,k+1\}$. But since over the Kronecker algebra we have $\Hom_H(R_p(n),R_q(m))=(0)$, for $p\neq q$ and any $m$ and $n$, we conclude that $\Ker\underline g$ does not contain any summand isomorphic to $R_p(n)$ for any $p$ and $n$. Hence, $\Ker \underline g$ only contains summands isomorphic to $S_n$ for some $n\geq (\sum_{\substack{i=1 \\ i\neq j}}^{k+1} n_i)-1$. The only $n$ such that $S_n$ satisfies the dimension formulas for the top and the socle sequences of the sequence $0\to S_{(\sum_{\substack{i=1 \\ i\neq j}}^{k+1} n_i)-1}\to \Ker \underline{g}\to R_{p_j(n_j)}\to 0$, is $n=(\sum_{i=1}^{k+1} n_i)-1$. Thus, we have that
\[\Ker \underline g\simeq S_{(\sum_{i=1}^{k+1} n_i)-1}.\]
Hence, we have shown that for any $m\geq 0$ we have an exact sequence 
\[0\to S_{(\sum_{i=1}^{m}n_i)-1}\to R_{p_1}(n_1)\oplus\cdots
\oplus R_{p_m}(n_m)\extto{\underline{g}=(g_1,\ldots,g_m)}S\to 0\]

Recall that $n_i$ is chosen to be the largest natural number such that $R_{p_i}(n_i)$ is isomorphic to a direct summand of $M$. Due to this fact it is not hard to see that the cokernel of the morphism 
\[\Hom_\L(M,\underline g)\colon \Hom_\L(M,R_{p_1}(n_1)\oplus\ldots\oplus R_{p_m}(n_m)) \to \Hom_\L(M,S)\]
is the simple module $\S_S$ (see again ~\cite[IIIV.7]{ARS} for the  structure of the $\Hom$-spaces between regular modules over the Kronecker algebra). Since the functor $\Hom_\L(M,-)$ is left exact, we have an exact sequence of $\Gamma$-modules
\begin{eqnarray}0\to\Hom_\L(M,S_{(\sum_{i=1}^{m}n_i)-1}) \to \Hom_\L(M, R_{p_1}(n_1)\oplus\cdots
\oplus R_{p_m}(n_m))\notag \\  \extto{\Hom_\L(M, \underline g)}\Hom_\L(M, S)\to \S_S\to 0 \notag
\end{eqnarray} 
Note that the $\Gamma$-modules $\Hom_\L(M, R_{p_1}(n_1)\oplus\cdots\oplus R_{p_m}(n_m))$ and $\Hom_\L(M, S)$ are projective, hence $\Hom_\L(M,S_{(\sum_{i=1}^{m}n_i)-1})\simeq \Omega_\Gamma^2(\S_S)$. Since we have assumed that if $m=1$, then $n_1>1$, we have that $S_{(\sum_{i=1}^{m}n_i)-1}$ is not in $\add M$, hence $\pd_\Gamma \S_S>2$. To compute the rest of the projective resolution of $\S_S$, we need to find an $\add M$-approximation of the $\L$-module $S_{(\sum_{i=1}^{m}n_i)-1}$. Let \[p\colon P\to S_{(\sum_{i=1}^{m}n_i)-1}\] be the projective cover of $S_{(\sum_{i=1}^{m}n_i)-1}$ and \[i\colon \Soc (S_{(\sum_{i=1}^{m}n_i)-1})\to S_{(\sum_{i=1}^{m}n_i)-1}\] be the natural inclusion. Then, it is easy to verify that the morphism 
\[P\oplus \Soc (S_{(\sum_{i=1}^{m}n_i)-1})\extto{\left(\begin{smallmatrix}p&i
\end{smallmatrix}\right)} S_{(\sum_{i=1}^{m}n_i)-1}\] is the minimal $\add M$-approximation of $S_{(\sum_{i=1}^{m}n_i)-1}$. Moreover, straightforward computation shows that 
\[\Ker (p \ i)\simeq S_{-(\sum_{i=1}^{m}n_i)+1}\]
Hence, we have an exact sequence
\begin{eqnarray}0\to\Hom_\L(M,S_{-(\sum_{i=1}^{m}n_i)+1})\to\Hom_\L(M, P\oplus \Soc (S_{(\sum_{i=1}^{m}n_i)-1}))\notag \\
\extto{\Hom_\L(M,\left(\begin{smallmatrix}p&i
\end{smallmatrix}\right))} \Hom_\L(M,S_{(\sum_{i=1}^{m}n_i)-1}) \to 0.\notag
\end{eqnarray}
\sloppy Since $S_{-(\sum_{i=1}^{m}n_i)+1}$ is not in $\add M$, we have that the $\Gamma$-module  $\Hom_\L(M,S_{-(\sum_{i=1}^{m}n_i)+1})$, which is isomorphic to $\Omega_\Gamma^3(\S_S)$, is not projective. Hence $\pd_\Gamma (\S_S)\geq 4$, which implies that $\gldim \Gamma\geq 4$. Thus, $\gldim \End_\L(M)\geq 4$.
\end {proof}

So, according to Propositions \ref{allreg} and \ref{oneprep}, the set of non-projective indecomposable summands of an Auslander generator  $M$ of $\L$, must contain at least two modules from the component of the AR-quiver that contains $\L$. We show that if $M$ is in addition minimal, then $M\simeq M_n$, for some integer $n$. We need the following lemma.

\begin{lem}\label{preinj}
For any positive integers $n$ and $k$, there exists a short exact sequence
\[0\to S_{-(n+k+1)}^k \to S_{-(n+k)}^{k+1}\extto{f_{-(n+k)}} S_{-n}\to 0,\] 
where the morphism $f_{-(n+k)}$ has the following property: if $X$ is an indecomposable $\L$-module which is not isomorphic to $S_{-(n+i)}$ for $i=0,\ldots,k-1$, then any morphism $f\colon X\to S_{-n}$, factors through $f_{-(n+k)}$. 
\end{lem}

\begin{proof}
We prove the claim using induction on $k$. Let $k=1$. Then the almost split sequence 
\[0\to S_{-(n+2)}\to S_{-(n+1)}^2\to S_{-n} \to 0\]
has the desired properties, so we can choose $f_{-(n+1)}$, to be the minimal right almost split morphism ending at $S_{-n}$. Assume that for $k=l$ there exists a short exact sequence
\[0\to S_{-(n+l+1)}^l \extto{g} S_{-(n+l)}^{l+1}\extto{f_{-(n+l)}} S_{-n}\to 0,\] 
with the property described in the statement of the lemma. Let 
\[0\to S_{-(n+l+2)}^{l+1}\to S_{-(n+l+1)}^{2(l+1)}\extto{\epsilon} S_{-(n+l)}^{l+1}\to 0\]
be the direct sum of $l+1$ copies of the almost split sequence ending at $S_{-(n+l)}$. Then, there exists a morphism $h\colon S_{-(n+l+1)}^l\to S_{-(n+l+1)}^{2(l+1)}$ such that $\epsilon\comp h =g$. Viewing these morphisms over the Kronecker algebra $H$, and using that $\dim_k\Hom_H(S_{-(n+l+1)},S_{-(n+l+1)})=1$, we conclude that $h$ is a split monomorphism. Hence, we obtain the following pullback diagram
\[\xymatrix@C=25pt{&&0\ar[d]&0\ar[d]&\\
&& S_{-(n+l+1)}^l \ar@{=}[r] \ar[d]^{h}& S_{-(n+l+1)}^l \ar[d]^{g}& \\
0\ar[r] &S_{-(n+l+2)}^{l+1}\ar[r]\ar@{=}[d]& S_{-(n+l+1)}^{2(l+1)}\ar[r]^{\epsilon}\ar[d]& S_{-(n+l)}^{l+1}\ar[d]^{f_{-(n+l)}}\ar[r]&0\\
0\ar[r]& S_{-(n+l+2)}^{l+1}\ar[r]&S_{-(n+l+1)}^{l+2}\ar[d]\ar[r]^{f_{-(n+l+1)}}&S_{-n}\ar[d]\ar[r]&0\\
&&0&0&}\]
It is easy to verify, from the above commutative diagram, that the morphism $f_{-(n+l+1)}$ has the desired factorization property.
\end{proof}
We are now ready to prove that the $\L$-modules $M_n=\L\oplus S_n\oplus S_{n+1}$, for $n$ in $\Z$, form a complete set of non-isomorphic minimal Auslander generators of $\L$.
\begin{thm}
Let $M$ be a minimal  Auslander generator of $\L$. Then $M$ is isomorphic to $M_n=\L\oplus S_n\oplus S_{n+1}$, for some integer $n$.
\end{thm}

\begin{proof}
Let \[M=\L\oplus M_1\oplus\cdots\oplus M_s,\] where $M_i$ is indecomposable non-projective for $i=1,\ldots,s$. In view of Propositions \ref{allreg} and \ref{oneprep} the set of the indecomposable non-projective direct summands of $M$ must contain at least two modules from the component of the AR-quiver of $\L$ that contains $\L$. So we can assume that $M_1\simeq S_m$ and $M_2\simeq S_n$, where $n$ and $m$ are such that $n<m$, and if one of the modules $M_i$, for $i=3,\ldots,s$, is isomorphic to $S_l$, then $l<n$. Let $i$ be an integer such that $m-2i<-1$ and consider the module 
\[\widetilde{M}=\L\oplus \tau^i(M_1\oplus\cdots\oplus M_s).\]
By Corollary \ref{dugascor}, we have that $\widetilde M$ is also a minimal  Auslander generator. 
Moreover \[\tau^i M_1\simeq \tau^i S_m\simeq S_{m-2i}\] and \[\tau^i M_2\simeq \tau^i S_n\simeq S_{n-2i}.\] Set $m-2i=m'$ and $n-2i=n'$. Then $n'<m'<-1$. In order to prove the claim of the theorem, we compute a projective resolution of the simple $\End_\L(\widetilde M)^{\op}$-module $\S_{\L}$, that corresponds to $\L$. We show that $\pd \S_\L\leq 3$ if and only if $m'=n'+1$. 

We first need to compute an $\add \widetilde M$-approximation of $\rad \L=S_{-1}$. Let
\[0\to S_{m'-1}^{-(m'+1)}\to S_{m'}^{-m'}\extto{f_{m'}} S_{-1}\to 0\]
be the short exact sequence that we get from Lemma \ref{preinj} for $n=-1$ and $k=-m'+1$. Then, by the choice of $m'$, the morphism $f_{m'}$ is a right $\add\widetilde M$-approximation of $S_{-1}$. Hence, applying the functor $\Hom_\L(\widetilde M, -)$ to the short exact sequence above, we get the short exact sequence
\[0\to\Hom_\L(\widetilde M, S_{m'-1}^{-(m'+1)})\to\Hom_\L(\widetilde M, S_{m'}^{-m'})\to \Hom_\L(\widetilde M, S_{-1})\to 0.\]
But $S_{-1}=\rrad_\L$ and the natural inclusion $i\colon S_{-1}\to \L$ is the right almost split morphism ending at $\L$, so we have a short exact sequence
\[0\to\Hom_\L(\widetilde M, S_{-1})\to \Hom_\L(\widetilde M, \L)\to \S_{\L}\to 0.\]
Moreover, $\L$ and $S_{m'}$ are in $\add \widetilde M$, which implies that $\Hom_\L(\widetilde M, \L)$ and $\Hom_\L(\widetilde M, S_{m'}^{-m'})$ are projective $\End_\L(\widetilde M)^{\op}$-modules. Hence, \[\Hom_\L(\widetilde M, S_{m'-1}^{-(m'+1)})=\Omega_{\End_\L(\widetilde M)^{\op}}^2(\S_{\L}).\]

Now, assume that $S_{m'-1}$ is not in $\add \widetilde M$. We will show that this assumption leads to a contradiction. To continue the projective resolution of $\S_\L$, we need to compute a right $\add \widetilde M$-approximation of $S_{m'-1}$. Let 
\[0\to S_{n'-1}^{m'-n'-1}\to S_{n'}^{m'-n'}\extto{f_{n'}} S_{m'-1}\to 0\]
be the short exact sequence that we get from Lemma \ref{preinj} for $n=-(m'-1)$ and $k=m'-n'-1$. Then, by the choice of $m'$ and $n'$, the morphism $f_{n'}$, is a right $\add \widetilde M$-approximation of $S_{m'-1}$. Hence, applying the functor $\Hom_\L(\widetilde M, -)$, we get the short exact sequence
\[0\to\Hom_\L(\widetilde M, S_{n'-1}^{m'-n'-1})\to\Hom_\L(\widetilde M, S_{n'}^{m'-n'})\to \Hom_\L(\widetilde M, S_{m'-1})\to 0.\]
Since $S_{n'}$ is in $\add \widetilde M$, the $\End_\L(\widetilde M)^{\op}$-module $\Hom_\L(\widetilde M, S_{n'}^{m'-n'})$ is projective. Hence,
\[\Hom_\L(\widetilde M, S_{n'-1}^{m'-n'-1})^{-(m'+1)}=\Omega_{\End_\L(\widetilde M)^{\op}}^3(\S_{\L}).\]

Since $\widetilde M$ is an Auslander generator, we have that $\gldim\End\L(\widetilde M)^{\op}=3$, so $\pd_{\End_\L(\widetilde M)^{\op}}\S_{\L}\leq 3$. Hence, $\Hom_\L(\widetilde M, S_{n'-1}^{m'-n'-1})^{-(m'+1)}$ is projective which implies that $S_{n'-1}$ is in $\add \widetilde M$. But then $\widetilde M$ contains as direct summands the modules $S_{n'-1}$, $S_{n'}$ and $S_{m'}$, which contradicts the minimality of $\widetilde M$, since the module $M_{n'-1}= \L\oplus\S_{n'-1}\oplus S_{n'}$ is already an Auslander generator.

Hence, the $\L$-module $S_{m'-1}$ is in $\add \widetilde M$, and since $\widetilde M$ is a minimal Auslander generator, we have that $m'=n'+1$
\[\widetilde M \simeq \L \oplus S_{n'}\oplus S_{n'+1}\simeq M_{n'}\]
and then, 
\[M\simeq \L\oplus S_n\oplus S_{n+1}\simeq M_{n}.\]
\end{proof}



We end this section with a remark on the number of the non-projective indecomposable direct summands of the minimal Auslander generators of $\L$. In ~\cite{R2}, Rouquier proved that for a selfinjective algebra $A$ with $\repdim A=3$, the number of the non-projective indecomposable direct summands $m$ of an Auslander generator $M$, is  at least half of the number of the isomorphism classes of simple $A$-modules $n$. In ~\cite{D}, Dugas improved this result by showing that $m\geq (2/3)n$, and if, in addition, $A$ is weakly symmetric or of Loewy length three, then $m\geq n$.  In the case of the exterior algebra in two variables $\L$, that we studied here, we have $m=2$, that is twice the number of simple $\L$-modules, for all minimal Auslander generators. This fact might suggest that the bounds given by Dugas can be further improved. Moreover, a natural question that arises is wether the number of the non-projective indecomposable direct summands is the same for all minimal Auslander generators, for any artin algebra.

\end{document}